\pdfoutput=1
\documentclass{article}
\usepackage{amsmath}
\usepackage{amsfonts}
\usepackage{amssymb}
\usepackage{amsthm}
\usepackage{bm}
\usepackage{mathrsfs}
\usepackage{stmaryrd}
\usepackage{tikz}
\usepackage{hyperref}
\usetikzlibrary{arrows,chains,matrix,positioning,scopes}
\usepackage[numbers]{natbib}
\usepackage{slashed}
\usepackage{todonotes}
\usepackage{graphicx}
\usepackage{microtype}
\usepackage{bbm}

\newtheorem{theorem}{Theorem}[section]
\newtheorem{lemma}[theorem]{Lemma}
\newtheorem{proposition}[theorem]{Proposition}

{\theoremstyle{definition} \newtheorem{definition}{Definition}}
{\theoremstyle{definition} }
{\theoremstyle{definition} \newtheorem*{exercise*}{Exercise}}
{\theoremstyle{definition}\newtheorem{example}{Example}}
{\theoremstyle{definition} }
{\theoremstyle{remark}}
{\theoremstyle{remark}}

{\theoremstyle{remark}}
{\theoremstyle{definition}}
{\theoremstyle{definition}}

\newcommand{\F}{\mathbb{F}}
\newcommand{\R}{\mathbb{R}}

\newcommand{\Z}{\mathbb{Z}}

\newcommand{\id}{\mathrm{id}}

\newcommand{\gr}{\mathrm{gr}}

\newcommand{\CKhred}{\widetilde{\mathcal{C}}_{\mathit{Kh}}}
\newcommand{\CKh}{\mathcal{C}_{\mathit{Kh}}}

\newcommand{\Kh}{\mathit{Kh}}
\newcommand{\Khred}{\widetilde{\mathit{Kh}}}
\newcommand{\rk}{\mathrm{rk}}
\newcommand{\KR}{\mathit{KR}}
\newcommand{\KRred}{\widetilde{\mathit{KR}}}

\newcommand{\aaii}{\,\raisebox{-0.3cm}{\includegraphics[scale=0.5]{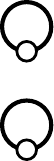}}\,}
\newcommand{\aaix}{\,\raisebox{-0.3cm}{\includegraphics[scale=0.5]{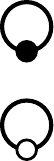}}\,}
\newcommand{\aaxi}{\,\raisebox{-0.3cm}{\includegraphics[scale=0.5]{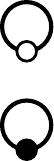}}\,}
\newcommand{\aaxx}{\,\raisebox{-0.3cm}{\includegraphics[scale=0.5]{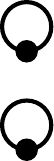}}\,}
\newcommand{\abi}{\,\raisebox{-0.3cm}{\includegraphics[scale=0.5]{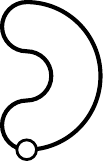}}\,}
\newcommand{\abx}{\,\raisebox{-0.3cm}{\includegraphics[scale=0.5]{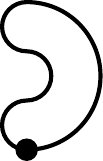}}\,}
\newcommand{\bai}{\,\raisebox{-0.3cm}{\includegraphics[scale=0.5]{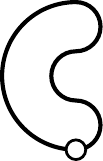}}\,}
\newcommand{\bax}{\,\raisebox{-0.3cm}{\includegraphics[scale=0.5]{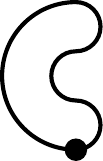}}\,}
\newcommand{\bbii}{\,\raisebox{-0.3cm}{\includegraphics[scale=0.5]{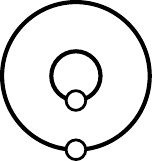}}\,}
\newcommand{\bbix}{\,\raisebox{-0.3cm}{\includegraphics[scale=0.5]{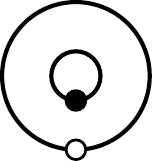}}\,}

\newcommand{\bbiii}{\,\raisebox{-0.59cm}{\includegraphics[scale=0.5]{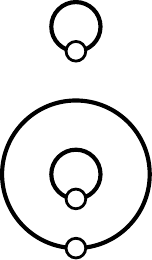}}\,}
\newcommand{\bbxii}{\,\raisebox{-0.59cm}{\includegraphics[scale=0.5]{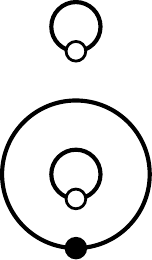}}\,}
\newcommand{\bbixi}{\,\raisebox{-0.59cm}{\includegraphics[scale=0.5]{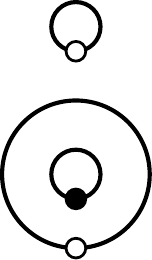}}\,}
\newcommand{\bbiix}{\,\raisebox{-0.59cm}{\includegraphics[scale=0.5]{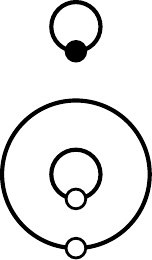}}\,}
\newcommand{\bbxxi}{\,\raisebox{-0.59cm}{\includegraphics[scale=0.5]{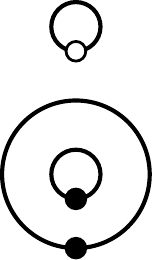}}\,}
\newcommand{\bbxix}{\,\raisebox{-0.59cm}{\includegraphics[scale=0.5]{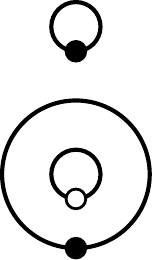}}\,}
\newcommand{\bbixx}{\,\raisebox{-0.59cm}{\includegraphics[scale=0.5]{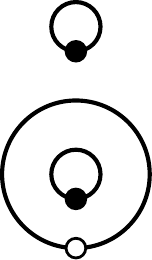}}\,}

\newcommand{\bei}{\,\raisebox{-0.59cm}{\includegraphics[scale=0.5]{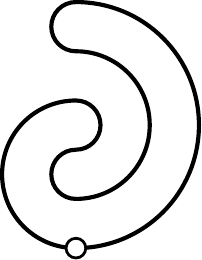}}\,}

\newcommand{\ccxxi}{\,\raisebox{-0.59cm}{\includegraphics[scale=0.5]{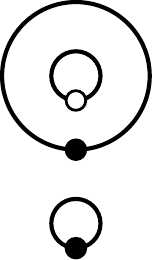}}\,}
\newcommand{\ccxix}{\,\raisebox{-0.59cm}{\includegraphics[scale=0.5]{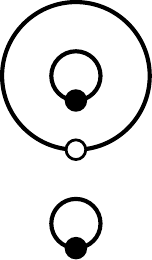}}\,}
\newcommand{\ccixx}{\,\raisebox{-0.59cm}{\includegraphics[scale=0.5]{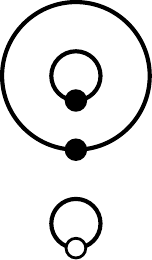}}\,}

\newcommand{\cei}{\,\raisebox{-0.59cm}{\includegraphics[scale=0.5]{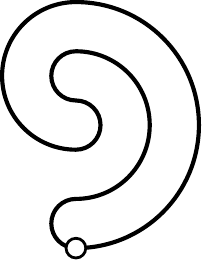}}\,}

\newcommand{\ddxii}{\,\raisebox{-0.59cm}{\includegraphics[scale=0.5]{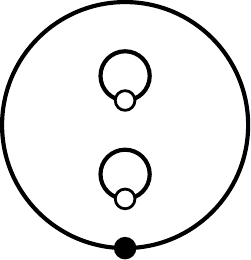}}\,}
\newcommand{\ddixi}{\,\raisebox{-0.59cm}{\includegraphics[scale=0.5]{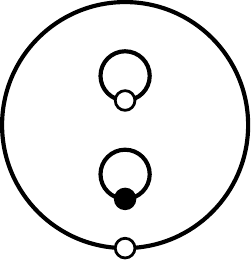}}\,}
\newcommand{\ddiix}{\,\raisebox{-0.59cm}{\includegraphics[scale=0.5]{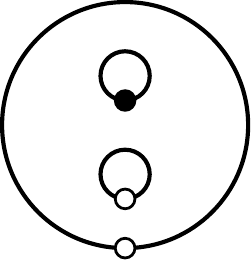}}\,}
\newcommand{\ddxxi}{\,\raisebox{-0.59cm}{\includegraphics[scale=0.5]{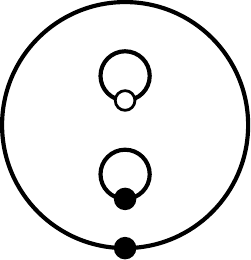}}\,}
\newcommand{\ddxix}{\,\raisebox{-0.59cm}{\includegraphics[scale=0.5]{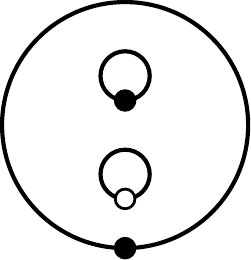}}\,}
\newcommand{\ddixx}{\,\raisebox{-0.59cm}{\includegraphics[scale=0.5]{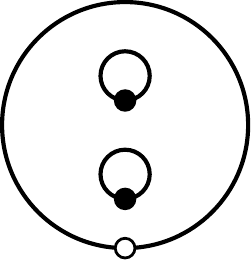}}\,}

\newcommand{\deii}{\,\raisebox{-0.59cm}{\includegraphics[scale=0.5]{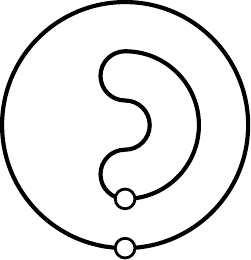}}\,}
\newcommand{\dexi}{\,\raisebox{-0.59cm}{\includegraphics[scale=0.5]{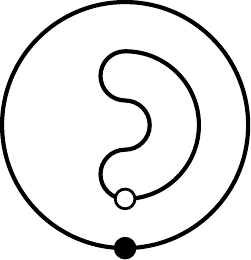}}\,}
\newcommand{\deix}{\,\raisebox{-0.59cm}{\includegraphics[scale=0.5]{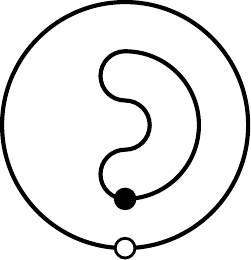}}\,}

\newcommand{\ebi}{\,\raisebox{-0.59cm}{\includegraphics[scale=0.5]{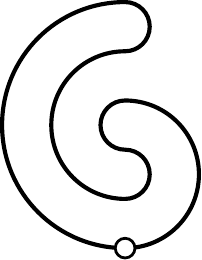}}\,}

\newcommand{\eci}{\,\raisebox{-0.59cm}{\includegraphics[scale=0.5]{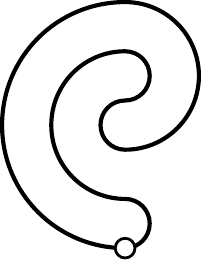}}\,}

\newcommand{\edii}{\,\raisebox{-0.59cm}{\includegraphics[scale=0.5]{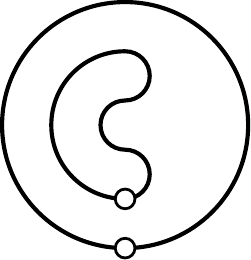}}\,}

\newcommand{\tangleone}{\,\raisebox{-0.3cm}{\includegraphics[scale=0.5]{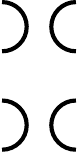}}\,}

\title{An Exceptional Splitting of Khovanov's Arc Algebras in Characteristic 2}
\author{Jesse Cohen\thanks{This material is based upon work supported by the National Science Foundation under Grant No. DMS-2204214.}\\University of Oregon}

\begin{document}
	\maketitle
	\begin{abstract}
		We show that there is an associative algebra $\widetilde{H}_n$ such that, over a base ring $R$ of characteristic 2, Khovanov's arc algebra $H_n$ is isomorphic to the algebra $\widetilde{H}_n[x]/(x^2)$. We also show a similar result for bimodules associated to planar tangles and prove that there is no such isomorphism over $\Z$.
	\end{abstract}
	\section{Introduction}
	In \cite{Khovanov2000}, Mikhail Khovanov introduced a categorification of the Jones polynomial in the form of a bigraded homology group
	\begin{align}
		\Kh(L)=\bigoplus_{i,j\in\Z}\Kh^{i,j}(L)
	\end{align}
	associated to each oriented link $L$ in $S^3$. This group has the unreduced Jones polynomial as its graded Euler characteristic:
	\begin{align}
		J(L)=\sum_{i,j\in\Z}(-1)^{i}\rk(\Kh^{i,j}(L))q^j.
	\end{align}
	These homology groups are functorial under smooth link cobordisms and have been used to great effect in low-dimensional topology. There are a variety of spectral sequences, many of which are themselves link invariants (cf. \cite{Baldwin2019}), whose $E^2$-pages are given by either Khovanov homology or its reduced version $\Khred(L)$ (cf. \cite{OzsSzBranched,Bloom2011,Kronheimer2011,Baldwin2019,Batson2015,Dowlin2018} for some examples). In \cite{Rasmussen2010}, Rasmussen used the spectral sequence defined by Lee in \cite{Lee2002} to define the $s$-invariant $s(K)$ of a knot $K$ and used this to give a combinatorial reproof of the Milnor conjecture --- that the slice genus of the $(p,q)$-torus knot is $\frac{(p-1)(q-1)}{2}$ --- the original proof of which, due to Kronheimer-Mrowka \cite{kronheimer1993gauge}, relied heavily on gauge theory. Similarly, the $s$-invariant can be used to give a combinatorial proof of the existence of exotic smooth structures on $\R^4$ (cf. \cite{Rasmussen2005}). More recently, the $s$-invariant was used by Piccirillo in \cite{Piccirillo2020} to show that the Conway knot is not smoothly slice and, in a similar vein, Hayden-Sundberg show in \cite{HaydenSundberg2021} that the cobordism maps on Khovanov homology can be used to distinguish exotically knotted smooth surfaces in the 4-ball which are topologically but not smoothly isotopic.
	
	Khovanov homology $\Kh(L)$ also admits a refinement to a spectrum  $\mathcal{X}_{\Kh}(L)$ (cf. \cite{LawsonLipshitzSarkar2021,HuKrizKriz2016}) whose homotopy type is an invariant of the link $L$ and whose reduced singular cohomology recovers $\Kh(L)$. This allows for the construction of Steenrod operations on Khovanov homology (cf. \cite{LipshitzSarkar2014,Bodish2020,Moran2019}) which can be used to distinguish some pairs of non-isotopic knots with the same Khovanov homology groups.
 	
	In \cite{Khovanov2002}, Khovanov defined algebras $H_n$, the \emph{arc algebras} on $2n$ points, and associated to an $(2m,2n)$-tangle diagram $T$ a complex of $(H_m,H_n)$-bimodules $\CKh(T)$ whose chain homotopy type is an invariant of the underlying tangle in $D^2\times I$. These bimodules and their variants can also be used to define invariants of annular links (cf. \cite{Beliakova2019,Lipshitz2020,Lawson2022}) as well as links in $S^2\times S^1$ (cf. \cite{Rozansky,Willis,MMSW}). The algebras $H_n$ and the bimodules $\CKh(T)$, like Khovanov homology, admit stable homotopy refinements (cf. \cite{LawsonLipshitzSarkar2021,Lawson2022}) and are also of importance in the representation theory of the quantum group $U_q(\mathfrak{sl}_2)$ (cf. \cite{ChenKhovanov2014,Stroppel2009,Brundan2010,Brundan2011}).
	\subsection{Results}
	In \cite{Shumakovitch2004TorsionOT}, Shumakovitch showed that Khovanov homology with co\"{e}fficients in $\F=\F_2$ decomposes as $\Kh(L)\cong\Khred(L)\otimes A$, where $A=\F[x]/(x^2)$. We show that the analog of this result holds for the arc algebras $H_n$ and Khovanov bimodules $\CKh(T)$ defined over a ring of characteristic 2. We also show that there is no such isomorphism of arc algebras over $\Z$.
	\subsection*{Acknowledgments}
	The author would like to thank Mikhail Khovanov, Robert Lipshitz, Catharina Stroppel, Nicolas Addington, and Ben Young for several helpful conversations. 
	\section{Background}
	Fix a base ring $R$. Given a non-negative integer $n$, let $\mathfrak{C}_n$ denote the set of planar crossingless matchings on $2n$ points, i.e. the set of $(2n,0)$ Temperley-Lieb diagrams.
	\begin{figure}
		\begin{center}
			\begin{align*}
				\mathfrak{C}_2=\left\{\,\raisebox{-0.7cm}{\includegraphics[scale=1]{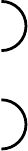}}\,\,\raisebox{-0.5cm}{,}\,\,\,\raisebox{-0.7cm}{\includegraphics[scale=1]{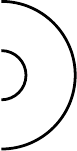}}\right\}
			\end{align*}
		\end{center}
		\caption{The set $\mathfrak{C}_2$ of planar crossingless matchings on $4$ points.}
		\label{C2}
	\end{figure}
	The \emph{arc algebra} $H_n$ over $R$ to be the associative graded $R$-algebra
	\begin{align}\label{HnDef}
		H_n=q^{-n}\bigoplus_{a,b\in\mathfrak{C}_n}\CKh(a^!b),
	\end{align}
	where $a^!$ is the result of flipping $a$ across the vertical axis, $a^!b$ is the result of gluing $a^!$ and $b$ along their common endpoints, and $\CKh:\mathrm{Cob}^{1+1}\to R-\textup{Mod}$ is Khovanov's TQFT whose value on a single circle is given by
	\begin{align}
		\CKh(\bigcirc)=A:=R[x]/(x^2)
	\end{align}
	as a commutative Frobenius algebra with comultiplication defined on generators by $\Delta(1)=1\otimes x+x\otimes 1$ and $\Delta(x)=x\otimes x$. The elements $1$ and $x$ are endowed with an integer-valued \emph{quantum grading} by taking $\gr_q(1)=1$ and $\gr_q(x)=-1$ and the formal power $q^{-n}$ in line (\ref{HnDef}) denotes a shift in this grading by $-n$. We take the convention that $H_0=R$. It is well-known that $H_n$ has the structure of a graded associative unital $R$-algebra given by applying the functor $\CKh$ to the minimal saddle cobordisms $\Sigma_{a,b,c}:a^!b\sqcup b^!c\to a^!c$. More precisely, if $\bm{v}$ and $\bm{v}'$ are labelings of the components of $a^!b$ and $b^!c$, respectively, then the product $(a^!b,\bm{v})(b^!c,\bm{v}')$ is given by $\Kh(\Sigma_{a,b,c})(\bm{v}\sqcup\bm{v}')$ and products of the form $(a^!b,\bm{v})(c^!d,\bm{v}')$ for $c\neq b$ vanish.
	\begin{definition}
		Given a crossingless matching $a\in\mathfrak{C}_n$, we distinguish the bottom-most of its $2n$ endpoints as a marked point. The \emph{reduced arc algebra} over $R$ on $2n$-points is then the associative graded $R$-algebra $\widetilde{H}_n$ defined by
		\begin{align}
			\widetilde{H}_n=q^{-n}\bigoplus_{a,b\in\mathfrak{C}_n}\CKhred(a^!b).
		\end{align}
		Here, $\CKhred$ denotes the reduced Khovanov complex given by the choice of basepoint as the quotient complex in which the marked component of every generator is labeled with a 1 and the entire complex is endowed with a quantum grading shift of $-1$.
	\end{definition}
	\begin{lemma}
		Let $\widetilde{m}:\widetilde{H}_n\otimes\widetilde{H}_n\to\widetilde{H}_n$ be the map induced by multiplication on $H_n$. Then $(\widetilde{H}_n,\widetilde{m})$ is a graded associative unital algebra.
	\end{lemma}
	\begin{proof}
		It is straightforward to see that the subgroup $I_x\subset H_n$ generated by elements in which the marked component is labeled by $x$ is a homogeneous two-sided ideal. The statement then follows from the fact that $\widetilde{H}_n=H_n/I_x$.
	\end{proof}
	\section{The Main Theorem}
	Given crossingless matchings $a,b\in\mathfrak{C}_n$, let $\kappa_0\in\pi_0(a^!b)$ be the marked component of $a^!b$ and define $\pi_*(a^!b)=\pi_0(a^!b)\smallsetminus\{\kappa_0\}$. We define a linear map $\lambda:\widetilde{H}_n\otimes A\to H_n$ as follows. Let
	\begin{align}
		\widetilde{\mathcal{B}}_n=\bigcup_{a,b\in\mathfrak{C}_n}\left\{(a^!b,\bm{v})|\bm{v}\in\{1,x\}^{\pi_*(a^!b)}\right\}
	\end{align}
	be the ``standard'' basis for $\widetilde{H}_n$ consisting of two crossingless matchings $a,b\in\mathfrak{C}_n$ and a labelling $\bm{v}:\pi_*(a^!b)\to\{1,x\}$ of the unmarked components of $a^!b$ by either $1$ or $x$. The marked component of a generator of $\widetilde{H}_n$ will always implicitly be labeled by $1$ but, in light of the following, it will be convenient to think of the labeling restricted to unmarked components only. Given a basis element $(a^!b,\bm{v})\in\widetilde{\mathcal{B}}_n$ and $s\in\{1,x\}$, let $(a^!b,\bm{v})_s\in H_n$ be the result of extending the labeling $\bm{v}$ to all of $\pi_0(a^!b)$ by taking $\bm{v}(\kappa_0)=s$. Now define
	\begin{align}
		\mathfrak{X}(a^!b,\bm{v})=\{\kappa\in\pi_*(a^!b)|\bm{v}(\kappa)=x\}
	\end{align}
	and, for a component $\kappa\in\mathfrak{X}(a^!b,\bm{v})$, define  $(a^!b,\bm{v}_\kappa)\in H_n$ by taking $\bm{v}_\kappa(\kappa_0)=x$, $\bm{v}_\kappa(\kappa)=1$, and $\bm{v}_\kappa(\kappa')=\bm{v}(\kappa')$ for all other components $\kappa'$. In other words, $(a^!b,\bm{v}_\kappa)$ is the result of labeling the marked component by $x$ and relabeling $\kappa$ with $1$. We then define $\lambda$ on basis elements $(a^!b,\bm{v})\otimes s\in\widetilde{\mathcal{B}}_n\otimes\{1,x\}$ by
	\begin{align}
		\lambda((a^!b,\bm{v})\otimes s)=\begin{cases}
			(a^!b,\bm{v})_x\hspace{3.6cm}\textup{if $s=x$}\\
			(a^!b,\bm{v})_1+\sum\limits_{\kappa\in\mathfrak{X}(a^!b,\bm{v})}(a^!b,\bm{v}_\kappa)\hspace{0.5cm}\textup{otherwise}.
		\end{cases}
	\end{align}
	\begin{example}
		Letting hollow and solid dots represent the labels of components via the convention $\circ=1$ and $\bullet=x$, if
		\begin{align}
			(a^!b,\bm{v})=\aaix
		\end{align}
		i.e. $a=b$ is the first of the crossingless matchings in $\mathfrak{C}_2$ depicted in Figure \ref{C2} and $\bm{v}:\pi_*(a^!b)\to\{1,x\}$ is the map taking the unmarked component of $a^!b$ to $x$, then
		\begin{align}
			\lambda((a^!b,\bm{v})\otimes 1)=\aaix+\aaxi
		\end{align}
		and
		\begin{align}
			\lambda((a^!b,\bm{v})\otimes x)=\aaxx.
		\end{align}
	\end{example}
	\begin{lemma}
		$\lambda$ is a graded $R$-linear isomorphism.
	\end{lemma}
	\begin{proof}
		Note that the set
		\begin{align}
			\mathcal{B}_n=\bigcup_{a,b\in\mathfrak{C}_n}\left\{\lambda((a^!b,\bm{v})\otimes s)|\bm{v}\in\{1,x\}^{\pi_*(a^!b)},s\in\{1,x\}\right\}
		\end{align}
		forms an $R$-basis for $H_n$ since there is a block lower-triangular matrix of the form
		\begin{align}
			\begin{pmatrix}
				\id & 0\\
				B & \id
			\end{pmatrix},
		\end{align}
		where $B$ is a square matrix with entries in $\{0,1\}$, taking the standard basis
		\begin{align}
			\mathcal{B}_n^{\hspace{0.02cm}\mathrm{std}}=\bigcup_{a,b\in\mathfrak{C}_n}\left\{(a^!b,\bm{v})|\bm{v}\in\{1,x\}^{\pi_0(a^!b)}\right\}
		\end{align}
		for $H_n$ to $\mathcal{B}_n$. Here, we order $\mathcal{B}_n^{\hspace{0.02cm}\mathrm{std}}$ so that those basis elements with $\bm{v}(\kappa_0)=1$ appear first in the ordering. Now we have $\mathrm{rk}_R\widetilde{H}_n\otimes_R A=\mathrm{rk}_R H_n$ so $\lambda$ is automatically an $R$-linear isomorphism since commutative rings have the invariant basis number property and both $\widetilde{H}_n\otimes_R A$ and $H_n$ are free as $R$-modules. Note that $\gr_q((a^!b,\bm{v})\otimes 1)=\gr_q((a^!b,\bm{v})_1)=\gr_q((a^!b,\bm{v}_\kappa))$ for any $\kappa\in\mathfrak{X}(a^!b,\bm{v})$ since each of these has the same number of tensor factors of $1$ and $x$. For the same reason, we have $\gr_q((a^!b,\bm{v})\otimes x)=\gr_q((a^!b,\bm{v})_x)$ so $\lambda$ preserves quantum gradings and is, therefore, a graded isomorphism.
	\end{proof}
	\begin{theorem}\label{MainThm}
		If $R$ is a ring of characteristic 2, then $\lambda$ is a graded $R$-algebra isomorphism.
	\end{theorem}
	\begin{proof}
		We have already shown that $\lambda$ is a graded linear isomorphism so it suffices to show that it is multiplicative, i.e. that
		\begin{align}
			\lambda((a^!b,\bm{v})\otimes s_1)\lambda((b^!c,\bm{v}')\otimes s_2)=\lambda((a^!b,\bm{v})(b^!c,\bm{v}')\otimes s_1s_2).
		\end{align}
		We do this by dividing into cases --- note that we do not need to consider products of the form $(a^!b,\bm{v})(c^!d,\bm{v}')$ for $b\neq c$ since these are always zero in $H_n$ and, therefore, also in $\widetilde{H}_n$.
		\renewcommand{\qedsymbol}{}
	\end{proof}
	\begin{proof}[Case 1: $s_1=s_2=x$]
		By far-commutation of saddles, we may always arrange for the marked components to merge first. Since $x^2=0$, we have
		\begin{align}
			(a^!b,\bm{v})_x(b^!c,\bm{v}')_x=0,
		\end{align}
		i.e. $0=\lambda(((a^!b,\bm{v})\otimes x)((b^!c,\bm{v}')\otimes x))=\lambda((a^!b,\bm{v})\otimes x)\lambda((b^!c,\bm{v}')\otimes x)$ for any basis elements $(a^!b,\bm{v}),(b^!c,\bm{v}')\in\widetilde{\mathcal{B}}_n$.
		\renewcommand{\qedsymbol}{}
	\end{proof}
	\begin{proof}[Case 2: $s_1=1$ and $s_2=x$]
		Next, consider $\lambda((a^!b,\bm{v})\otimes 1)\lambda((b^!c,\bm{v}')\otimes x)$: this is equal to $(a^!b,\bm{v})_1(b^!c,\bm{v}')_x$ since $(a^!b,\bm{v}_\kappa)(b^!c,\bm{v}')_x=0$ for any $\kappa\in\mathfrak{X}(a^!b,\bm{v})$ as the marked components of both elements in this product are labeled $x$ so their merger creates a label of $x^2=0$. Now suppose that the product $(a^!b,\bm{v})(b^!c,\bm{v}')$ in $\widetilde{H}_n$ is given as a linear combination of elements of the basis $\widetilde{\mathcal{B}}_n$ by
		\begin{align}
			(a^!b,\bm{v})(b^!c,\bm{v}')=\sum\limits_i(a^!c,\bm{v}''_i).
		\end{align}
		We claim that
		\begin{align}
			(a^!b,\bm{v})_1(b^!c,\bm{v}')_x=\sum\limits_i(a^!c,\bm{v}''_i)_x=\sum_i\lambda((a^!c;\bm{v}_i'')\otimes x).
		\end{align}
		Note that, under the saddle cobordism $a^!b\sqcup b^!c\to a^!c$, if the marked components merge and do not subsequently split, then this is true automatically. Otherwise, in $\widetilde{H}_n$, any splittings of the marked component produces some number of new components in the summands $(a^!c,\bm{v}''_i)$, each of which is labeled $x$. In $H_n$, after the first merger occuring in the saddle cobordism, the marked component of $(a^!b,\bm{v})_1(b^!c,\bm{v}')_x$ becomes labeled by $x$ and any subsequent splittings produce the same new components as before, each of which is again labeled by $x$ since we have $\Delta(x)=x\otimes x$. Therefore, we have
		\begin{align}
			\lambda((a^!b,\bm{v})\otimes 1)\lambda((b^!c,\bm{v}')\otimes x)=\lambda(((a^!b,\bm{v})\otimes 1)((b^!c,\bm{v}')\otimes x)),
		\end{align}
		as desired.
		\renewcommand{\qedsymbol}{}
	\end{proof}
	\begin{proof}[Case 3: $s_1=x$ and $s_2=1$]
		It follows from the previous case that
		\begin{align}
			\lambda((a^!b,\bm{v})\otimes x)\lambda((b^!c,\bm{v}')\otimes 1)=\lambda(((a^!b,\bm{v})\otimes x)((b^!c,\bm{v}')\otimes 1)).
		\end{align}
		To see this, note that the algebra anti-automorphisms $\overline{(-)}:\widetilde{H}_n\otimes A\to\widetilde{H}_n\otimes A$ and $\overline{(-)}:H_n\to H_n$ given in both cases by $\overline{(a^!b,\bm{v})}=(b^!a,\bm{v})$ satisfy $\overline{\lambda((a^!b,\bm{v})\otimes s)}=\lambda(\overline{(a^!b,\bm{v})}\otimes s)$ by construction.
		It is then straightforward to show that $\overline{\lambda((a^!b,\bm{v})\otimes x)\lambda((b^!c,\bm{v}')\otimes 1)}$ is equal to $\overline{\lambda(((a^!b,\bm{v})\otimes x)((b^!c,\bm{v}')\otimes 1))}$ by a direct computation using Case 2.
		\renewcommand{\qedsymbol}{}
	\end{proof}
	\begin{proof}[Case 4: $s_1=s_2=1$]
		Let $\Sigma:c\to c'$ be a connected, orientable, 2-dimensional cobordism, where $c$ and $c'$ are disjoint unions of planar circles. Recall that if $\bm{v}$ and $\bm{w}$ are labelings of $c$ and $c'$ by $\{1,x\}$ then $\bm{w}$ occurs as a summand in $\Kh(\Sigma)(\bm{v})$ if and only if $g(\Sigma)=0$ and $\#_x\bm{v}+\#_1\bm{w}=1$. Here, for a labeling $\bm{u}$, the quantities $\#_1\bm{u}$ and $\#_x\bm{u}$ are the number of components labeled $1$ and $x$ by $\bm{u}$. The same holds true for $\Khred(\Sigma)(\bm{v})$ subject to the constraint that only those $\bm{v}$ and $\bm{w}$ which label the marked component 1 are permitted (cf. Figure \ref{Saddle}). Now suppose we are given generators $(a^!b,\bm{v})$ and $(b^!c,\bm{v}')$ of $\widetilde{H}_n$ and consider the minimal saddle cobordism $\Sigma:a^!b\sqcup b^!c\to a^!c$. We claim that $\lambda(\Khred(\Sigma)(\bm{v}\sqcup\bm{v}')\otimes 1)=\Kh(\Sigma)(\lambda\otimes\lambda((\bm{v}\sqcup\bm{v}')\otimes 1))$.
		\begin{figure}
			\begin{center}
				\includegraphics[scale=0.5]{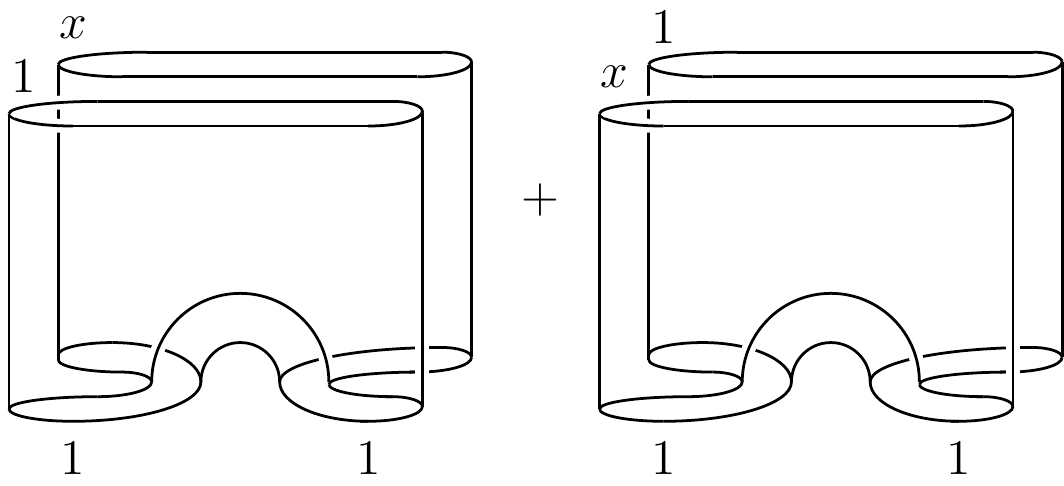}\hspace{2cm}\includegraphics[scale=0.5]{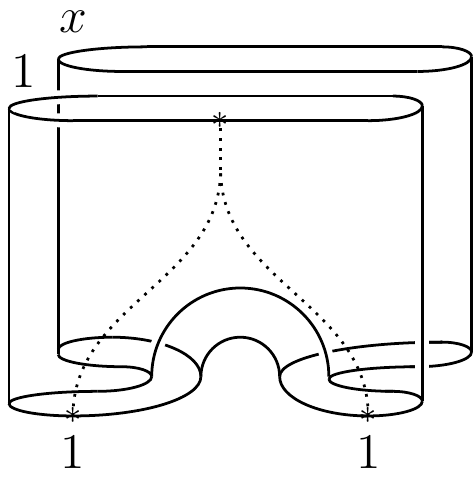}
			\end{center}
			\caption{An example of $\Kh(\Sigma)(\bm{v})$ (left) versus $\Khred(\Sigma)(\bm{v})$ (right) in which the two differ.}
			\label{Saddle}
		\end{figure}
		\begin{proof}[Subcase 1.]
			We first consider the case that $\bm{v}\sqcup\bm{v}'$ labels all of the incoming circles of the component $\Sigma_*$ of $\Sigma$ which contains the marked incoming circles by $1$. In the reduced product $\Khred(\Sigma)(\bm{v}\sqcup\bm{v}')$, each $\bm{w}$ occurring as a summand labels the marked outgoing circle by $1$ and any other outgoing circles of $\Sigma_*$ by $x$. The unreduced product $\Kh(\Sigma)(\bm{v}\sqcup\bm{v}')$ consists of these terms plus terms in which the marked outgoing circle is labeled $x$ and exactly one of the remaining outgoing circles of $\Sigma_*$ is labeled $1$. The summand of $\lambda(\Khred(\Sigma)(\bm{v}\sqcup\bm{v}')\otimes 1)$ consisting of $\Khred(\Sigma)(\bm{v}\sqcup\bm{v}')$ and those terms obtained only by summing over the $x$-labeled outgoing circles of $\Sigma_*$ is precisely $\Kh(\Sigma)(\bm{v}\sqcup\bm{v}')$. It thus suffices to show that the remaining terms either come in cancelling pairs or come from swapping the label on a marked incoming circle of $\Sigma$ with that of an $x$-labeled circle. Consider a connected component $\Sigma_1$ of $\Sigma\smallsetminus\Sigma_*$. If the incoming circles of $\Sigma_1$ are all labeled $1$ and $\Sigma_1$ has $\ell$ outgoing circles, then any labeling $\bm{w}_1$ of these circles occuring as a sublabeling of a term in $\Khred(\Sigma)(\bm{v}\sqcup\bm{v}')$ labels $\ell-1$ of them by $x$ and one of them by $1$. Moreover, if $\bm{w}$ is a summand of $\Khred(\Sigma)(\bm{v}\sqcup\bm{v}')$ and $\bm{w}_1$ occurs as a sublabeling of $\bm{w}$, then every possible labeling $\bm{w}'$ obtained from $\bm{w}$ by permutating $\bm{w}_1$ occurs exactly once as a summand of $\Khred(\Sigma)(\bm{v}\sqcup\bm{v}')$. Now, for any labeling $\bm{w}$ and sub-labeling $\bm{w}_1$ of the outgoing circles of $\Sigma_1$ and any choice of $x$-labeled component $\kappa$ coming from $\bm{w}_1$, there exists a $\bm{w}'$ and $\bm{w}_1'$ such that $\bm{w}$ and $\bm{w}'$ agree away from $\bm{w}_1$ and $\bm{w}_1'$ and a choice of $x$-labeled component $\kappa'$ coming from $\bm{w}_1'$ such that the labelings $\bm{w}_\kappa$ and $\bm{w}'_{\kappa'}$ agree. All such choices come in pairs so the summands of $\lambda(\Khred(\Sigma)(\bm{v}\sqcup\bm{v}')\otimes 1)$ coming from summing over the $x$-labeled outgoing circles of $\Sigma_1$ cancel.
			
			Note that if more than one incoming circle of $\Sigma_1$ is labeled $x$, then we have $\Khred(\Sigma)(\bm{v}\sqcup\bm{v}')=0$. On the other hand, we also have $\Kh(\Sigma)(\lambda\otimes\lambda((\bm{v}\sqcup\bm{v}')\otimes 1))=0$ since either more than two of the incoming circles is labeled $x$ --- in which case applying $\Kh(\Sigma)$ to every term of $\lambda\otimes\lambda((\bm{v}\sqcup\bm{v}')\otimes 1)$ yields zero --- or exactly two are, call them $\kappa$ and $\kappa'$. In the latter case, the terms $\Kh(\Sigma)((\bm{v}\sqcup\bm{v}')_\kappa)$ and $\Kh(\Sigma)((\bm{v}\sqcup\bm{v}')_{\kappa'})$ agree and, hence, cancel modulo 2. If exactly one of the incoming circles $\kappa_0$ of $\Sigma_1$ is labeled by $x$, then every outgoing circle of $\Sigma_1$ is also labeled $x$. If, as before, $\Sigma_1$ has $\ell$ outgoing circles $\kappa_1,\dots,\kappa_\ell$, then the summand $\Khred(\Sigma)(\bm{v}\sqcup\bm{v}')_{\kappa_1}+\cdots+\Khred(\Sigma)(\bm{v}\sqcup\bm{v}')_{\kappa_\ell}$ of $\lambda(\Khred(\Sigma)(\bm{v}\sqcup\bm{v}')\otimes 1)$ co\"{i}ncides precisely with the summand $\Kh(\Sigma)((\bm{v}\sqcup\bm{v}')_{\kappa_0})$ of $\Kh(\Sigma)(\lambda\otimes\lambda((\bm{v}\sqcup\bm{v}')\otimes 1))$. Therefore, we have $\lambda(\Khred(\Sigma)(\bm{v}\sqcup\bm{v}')\otimes 1)=\Kh(\Sigma)(\lambda\otimes\lambda((\bm{v}\sqcup\bm{v}')\otimes 1))$.
			\renewcommand{\qedsymbol}{}
		\end{proof}
		\begin{proof}[Subcase 2.]
			If at least one of the incoming circles of the component $\Sigma_*$ is labeled $x$, then $\Khred(\Sigma)(\bm{v}\sqcup\bm{v}')\otimes 1$ necessarily vanishes. If more than one of these incoming circles is labeled $x$, then, as before, every term of $\lambda((\bm{v}\sqcup\bm{v}')\otimes 1)$ necessarily also labels at least two of the incoming circles on this component by $x$ so we also have that $\Kh(\Sigma)(\lambda\otimes\lambda((\bm{v}\sqcup\bm{v}')\otimes 1))=0$. If exactly one incoming circle $\kappa_0$ of $\Sigma_*$ is labeled $x$ --- assume for simplicity that this label comes from $\bm{v}$ --- then the terms of $\lambda\otimes\lambda((\bm{v}\sqcup\bm{v}')\otimes 1)$ consist of $\bm{v}\sqcup\bm{v}'$, $\bm{v}_{\kappa_0}\sqcup\bm{v}'$, and terms of the form $\bm{v}_\kappa\sqcup\bm{v}'$, $\bm{v}\sqcup\bm{v}'_{\kappa'}$, and $\bm{v}_{\kappa}\sqcup\bm{v}'_{\kappa'}$ where $\kappa$ and $\kappa'$ are incoming circles of a component of $\Sigma\smallsetminus\Sigma_*$ labeled $x$ by $\bm{v}$ and $\bm{v}'$, respectively. In $\Kh(\Sigma)(\lambda((\bm{v}\sqcup\bm{v}')\otimes 1))$, the first two of these terms contribute two identical and hence cancelling terms since we are working in characteristic 2 and the remaining terms contribute $0$ since $\Sigma$ merges at least two $x$-labeled circles in those cases.
			\renewcommand{\qedsymbol}{}
		\end{proof}
	\end{proof}
	\begin{example}
		Using the same convention for hollow and filled dots as before, in $\widetilde{H}_3\otimes A$, we have
		\begin{align}
			\left(\cei\otimes 1\right)\left(\eci\otimes 1\right)=\ccixx\otimes 1,
		\end{align}
		while in $H_3$, we have
		\begin{align}
			\cei\eci=\ccixx+\ccxix+\ccxxi
		\end{align}
		and
		\begin{align}
			\ccixx+\ccxix+\ccxxi=\ccixx_1+\sum_{\kappa\in\mathfrak{X}\left(\scalebox{0.45}{\ccixx}\right)}\ccixx_\kappa=\lambda\left(\ccixx\otimes 1\right).
		\end{align}
	\end{example}
	\begin{example}
		In $\widetilde{H}_2\otimes A$, we have
		\begin{align}
			\left(\aaix\otimes 1\right)\left(\abi\otimes 1\right)=0
		\end{align}
		while
		\begin{align}
			\begin{split}
				\lambda\left(\aaix\otimes 1\right)\lambda\left(\abi\otimes 1\right)&=\left(\aaix+\aaxi\right)\abi\\&=2\abx\\&=0
			\end{split}
		\end{align}
		modulo 2, which shows that $\lambda$ cannot possibly be a multiplicative map in characteristics other than 2.
	\end{example}
	\begin{example}
		We consider two more examples to exhibit some of the phenomena that can occur when comparing $m\circ(\lambda\otimes\lambda)$ and $\lambda\circ \widetilde{m}$ in characteristic 2. Suppose that
		\begin{align}
			(a_1^!b_1,\bm{v})=\deii
		\end{align}
		and
		\begin{align}
			(b_1^!c_1,\bm{v}')=\edii,
		\end{align}
		then
		\begin{align}
			((a_1^!b_1,\bm{v})\otimes 1)((b_1^!c_1,\bm{v}')\otimes 1)=\left(\ddiix+\ddixi\right)\otimes 1
		\end{align}
		so
		\begin{align}
			\begin{split}
				\lambda(((a_1^!b_1,\bm{v})\otimes 1)((b_1^!c_1,\bm{v}')\otimes 1))&=\ddiix+\ddixi+2\ddxii\\[-0.3cm]\\&=\ddiix+\ddixi
			\end{split}
		\end{align}
		modulo 2. On the other hand, we have
		\begin{align}
			\begin{split}
				\lambda((a_1^!b_1,\bm{v})\otimes 1)\lambda((b_1^!c_1,\bm{v}')\otimes 1)&=\deii\edii\\[-0.3cm]\\&=\ddiix+\ddixi.
			\end{split}
		\end{align}
		This is an instance of the first part of Case 4, Subcase 1, in the proof of the main theorem. Similarly, if 
		\begin{align}
			(a_2^!b_2,\bm{w})=\deix
		\end{align}
		and $(b_2^!c_2,\bm{w}')=(b_1^!c_1,\bm{v}')$, then we have
		\begin{align}
			((a_2^!b_2,\bm{w})\otimes 1)((b_2^!c_2,\bm{w}')\otimes 1)=\ddixx\otimes 1
		\end{align}
		so
		\begin{align}
			\lambda(((a_2^!b_2,\bm{w})\otimes 1)((b_2^!c_2,\bm{w}')\otimes 1))=\ddixx+\ddxix+\ddxxi
		\end{align}
		while
		\begin{align}
			\lambda((a_2^!b_2,\bm{w})\otimes 1)=\deix+\dexi
		\end{align}
		so
		\begin{align}
			\begin{split}
				\lambda((a_2^!b_2,\bm{w})\otimes 1)\lambda((b_2^!c_2,\bm{w}')\otimes 1)&=\left(\deix+\dexi\right)\edii\\[-0.3cm]\\&=\ddixx+\ddxix+\ddxxi.
			\end{split}
		\end{align}
		This is an instance of the second part of Case 4, Subcase 1.
	\end{example}
	\subsection{Bimodules of planar tangles}
	Now suppose that $T$ is a planar (crossingless) $(2m,2n)$-tangle diagram and let
	\begin{align}
		\CKh(T)=\bigoplus_{a\in\mathfrak{C}_m,b\in\mathfrak{C}_n}\CKh(a^!Tb)
	\end{align}
	be the associated $(H_m,H_n)$-bimodule. Choose either the left bottom-most endpoint or the right bottom-most endpoint of $T$ as a marked point for every $a^!Tb$ and denote the corresponding reduced bimodules by $\CKhred^L(T)$ and $\CKhred^R(T)$, respectively. We define a map $\lambda^L:\CKhred^L(T)\otimes A\to\CKh(T)$ as follows: given a labeling $\bm{v}:\pi_*(a^!Tb)\to\{1,x\}$, let $\mathfrak{X}(a^!Tb,\bm{v})$ denote the set of all components of $a^!Tb$ labeled $x$ by $\bm{v}$. We then define
	\begin{align}
		\lambda^L((a^!Tb,\bm{v})\otimes s)=\begin{cases}
			(a^!Tb,\bm{v})_x\hspace{4.2cm}\textup{if $s=x$}\\
			(a^!Tb,\bm{v})_1+\sum\limits_{\kappa\in\mathfrak{X}(a^!Tb,\bm{v})}(a^!Tb,\bm{v}_\kappa)\hspace{0.5cm}\textup{otherwise},
		\end{cases}
	\end{align}
	where, as before, $(a^!Tb,\bm{v})_s$ and $(a^!Tb,\bm{v}_\kappa)$ are the elements of $\CKh(T)$ obtained by labeling the marked component by $s$ and by swapping the label of $\kappa$ and the marked component, respectively. We define $\lambda^R$ similarly. Note that if the bottom left-most and bottom right-most endpoints of $T$ are on the same connected component, then the two maps co\"{i}ncide.
	\begin{proposition}\label{BimoduleProp}
		If $R$ is a ring of characteristic $2$, then $\lambda^L$ (resp. $\lambda^R$) is a graded linear isomorphism intertwining the left $\widetilde{H}_m\otimes A$- and $H_m$-module (resp. right $\widetilde{H}_n\otimes A$- and $H_n$-module) structures on $\CKhred^L(T)\otimes A$ (resp. $\CKhred^R(T)\otimes A$) and $\CKh(T)$. However, they are not bimodule isomorphisms in general.
	\end{proposition}
	\begin{proof}
		The proof for both is essentially identical to the proof of Theorem \ref{MainThm}. The example that follows shows that $\lambda^L$ and $\lambda^R$ need not be bimodule isomorphisms when they are not equal.
	\end{proof}
	\begin{example}
		Let $T=\tangleone$ and consider $\aaix\aaxi\in\CKhred^L(T)$. Consider the left- and right-actions of the elements $\bai,\abi\in\widetilde{H}_2$: we have that
		\begin{align}
			\left(\bai\otimes 1\right)\left(\aaix\aaxi\otimes 1\right)=0
		\end{align}
		and
		\begin{align}
			\begin{split}
				\lambda^L\left(\bai\otimes 1\right)\lambda^L\left(\aaix\aaxi\otimes 1\right)&=\bai\left(\aaix\aaxi+\aaxi\aaxi+\aaxx\aaii\right)\\&=2\bax\aaxi\\&=0
			\end{split}
		\end{align}
		modulo 2, as expected, and, on the other hand, we have
		\begin{align}
			\left(\aaix\aaxi\otimes 1\right)\left(\abi\otimes 1\right)=0
		\end{align}
		while
		\begin{align}
			\begin{split}
				\lambda^L\left(\aaix\aaxi\otimes 1\right)\lambda^L\left(\abi\otimes 1\right)&=\left(\aaix\aaxi+\aaxi\aaxi+\aaxx\aaii\right)\abi\\&=\aaix\abx+\aaxi\abx+\aaxx\abi\\&\neq 0
			\end{split}
		\end{align}
		so $\lambda^L$ is not a right-module homomorphism, even in characteristic 2.
	\end{example}
	\section{$\Z$-co\"{e}fficients}
	We will now show that there is in general no such decomposition of arc algebras over $\Z$. To that end, let
	\begin{align}
		\alpha=a\aaii+b\aaix+c\abi+d\bai+e\bbii+f\bbix
	\end{align}
	be an arbitrary central element in $\widetilde{H}_2$. Then we have
	\begin{align}
		0=\left[\alpha,\bbii\right]=c\abi-d\bai,
	\end{align}
	so $c=d=0$. It then follows that
	\begin{align}
		0=\left[\alpha,\abi\right]=(a-e)\abi
	\end{align}
	so $a=e$. Therefore, $\alpha$ is of the form
	\begin{align}
		\alpha=a\left(\aaii+\bbii\right)+b\aaix+d\bbix.
	\end{align}
	One can check that both $\aaix$ and $\bbix$ are themselves central so
	\begin{align}
		Z(\widetilde{H}_2)=\Z\left\langle\aaii+\bbii,\aaix,\bbix\right\rangle.
	\end{align}
	Now, since $\widetilde{H}_2$ and $A$ are both free as abelian groups and $A$ is commutative, we have $Z(\widetilde{H}_2\otimes A)=Z(\widetilde{H}_2)\otimes A$.
	
	In \cite{Khovanov2006}, Khovanov showed that the only invertible central elements of degree 0 in $H_n$ with $\Z$-co\"{e}fficients are $\pm 1$ and, as a consequence, that if $M$ is an invertible complex of graded $H_n$-bimodules, then the only degree 0 automorphisms of $M$ are $\pm\id$. The same argument holds, mutatis mutandis, in characteristic 2 to show that the only degree 0 automorphisms of $\widetilde{H}_n\otimes A$ and $H_n$ are the respective identity maps. In particular, this tells us that if there were a graded algebra isomorphism $\Lambda:\widetilde{H}_2\otimes A\to H_2$, then $\Lambda=\lambda$ modulo 2 so
	\begin{align}
		\Gamma:=\Lambda\left(\aaix\otimes 1\right)=s\aaix+t\aaxi
	\end{align}
	for some $s,t\in\{\pm 1\}$. Now $\Gamma$ is central since $\Lambda$ is an algebra isomorphism so
	\begin{align}
		0=\left[\Gamma,\abi\right]=(s+t)\abx
	\end{align}
	and hence $t=-s$. Up to composing $\Lambda$ with $-\id$, we may assume $s=1$ so
	\begin{align}
		\Gamma=\aaix-\aaxi.
	\end{align}
	On the other hand, we have
	\begin{align}
		\left(\aaix\otimes 1\right)^2=0
	\end{align}
	so we would have to have
	\begin{align}
		0=\Gamma^2=-2\aaxx,
	\end{align}
	which is false. Therefore no such isomorphism can exist. Now note that $\widetilde{H}_2\otimes A$ and $H_2$ include into $\widetilde{H}_n\otimes A$ and $H_n$, respectively, as subalgebras $\widetilde{J}_2\otimes A$ and $J_2$ for any $n>2$ by stacking $n-2$ round 1-labeled circles above every generator. If we had a $\Z$-algebra isomorphism $\Lambda:\widetilde{H}_n\otimes A\to H_n$ and $e\in\widetilde{H}_n$ is a minimal idempotent, i.e. $e=(a^!a,\bm{1})$ for some $a\in\mathfrak{C}_n$, then we necessarily have that $\Lambda(e\otimes 1)=\pm e_1$ since $\lambda(e\otimes 1)=e_1$. This tells us that the restriction of $\Lambda$ to $\widetilde{J}_2\otimes A$ would give us an algebra isomorphism $\widetilde{J}_2\otimes A\to J_2$ but we have shown this is impossible. Therefore, there is no graded $\Z$-algebra isomorphism $\widetilde{H}_n\otimes A\to H_n$ for any $n>1$.
	\section*{Further Directions}
	In \cite{Wang2021}, Wang showed that there are bigraded $R$-module isomorphisms $\KR_p(L;R)\cong\KRred_p(L;R)\otimes_R R[x]/(x^p)$ relating the unreduced and reduced Khovanov-Rozansky $\mathfrak{sl}_p$-link homologies (cf. \cite{Khovanov2004a,Khovanov2008}) whenever $R$ is a ring of characteristic $p$. Analogs of the arc algebras in the setting of $\mathfrak{sl}_p$ homology, the $\mathfrak{sl}_p$-web algebras, were introduced by Mackaay-Pan-Tubbenhauer, in the $p=3$ case, and Mackaay in \cite{Mackaay2012,Mackaay2014}. There is also an annular version of the arc algebra which was studied by Ehrig-Tubbenhauer in \cite{Ehrig2019}.
	
	In \cite{Ozsvath2013}, Ozsv\'{a}th, Rasmussen, and Szab\'{o} defined an ``odd'' version of Khovanov homology using an exterior version of the Frobenius algebra used in the original construction. This invariant also categorifies the Jones polynomial and agrees with ordinary Khovanov homology modulo 2. As in the characteristic 2 case, there is a splitting of odd Khovanov homology with $\Z$-co\"{e}fficients (cf. \cite{Ozsvath2013}, Proposition 1.8). Moreover, other properties of Khovanov homology in characteristic 2 can be realized as the mod 2 reduction of a property of odd Khovanov homology. For instance, Wehrli proved in \cite{wehrli2010mutation} that Khovanov homology with $\F$-co\"{e}fficients is mutation invariant and this was shown by Bloom for odd Khovanov homology in \cite{bloom2009odd}. The odd analogues of the arc algebras and bimodules for tangles were studied by Naisse-Vaz in \cite{naisse2016odd} and Naisse-Putyra in \cite{naisse2020odd}, respectively. Unlike the ordinary arc algebras, however, odd arc algebras are only associative up to a sign depending on the elements being multiplied.
	
	In \cite{KhovanovRobert2020}, Khovanov and Robert studied a deformation $A_\alpha$ of the TQFT $A$ defined over the ring $R_\alpha=\Z[\alpha_0,\alpha_1]\cong H^*_{U(1)\times U(1)}(\mathrm{pt})$ as an $R_\alpha$-algebra by
	\begin{align}
		A_\alpha=R_\alpha[x]/((x-\alpha_0)(x-\alpha_1))\cong H^*_{U(1)\times U(1)}(S^2)
	\end{align}
	with comultiplication given by
	\begin{align}
		\begin{split}
			&1\mapsto 1\otimes x+x\otimes 1-(\alpha_0+\alpha_1)1\otimes 1\\
			&x\mapsto x\otimes x-\alpha_0\alpha_1 1\otimes 1.
		\end{split}
	\end{align}
	This TQFT defines a link invariant in the same way as does $A$ and, taking different values for the parameters $\alpha_0$ and $\alpha_1$ at the chain level, one can recover both Khovanov and Lee homology. One may define deformed arc algebras $H_n^\alpha$ and $\widetilde{H}_n^\alpha$ analogous to the unsual ones. However, even in characteristic 2, the na\"{i}ve $R_\alpha$-linear extension of $\lambda$ to a map $\widetilde{H}_n^\alpha\otimes A_\alpha\to H_n^\alpha$ is not multiplicative. For example, letting $h=\alpha_0+\alpha_1$ and $t=\alpha_0\alpha_1$ for the sake of brevity, in $\widetilde{H}_3^\alpha\otimes A_\alpha$, we have
	\begin{align}
		\begin{split}
			&\left(\bei\otimes 1\right)\left(\ebi\otimes 1\right)\\&=\left(h^2\bbiii+h\left(\bbixi+\bbiix\right)+\bbixx\right)\otimes 1
		\end{split}
	\end{align}
	so
	\begin{align}
		\begin{split}
			&\lambda\left(\left(\bei\otimes 1\right)\left(\ebi\otimes 1\right)\right)\\&=h^2\bbiii+h\left(\bbixi+\bbiix\right)+\bbixx+\bbxix+\bbxxi.
		\end{split}
	\end{align}
	On the other hand, in $H_3^\alpha$, we have
	\begin{align}
		\begin{split}
			&\bei\ebi\\&=(h^2+t)\bbiii+h\left(\bbixi+\bbiix+\bbxii\right)+\bbixx+\bbxix+\bbxxi
		\end{split}
	\end{align}
	so
	\begin{align}
		\lambda\left(\left(\bei\otimes 1\right)\left(\ebi\otimes 1\right)\right)\neq\lambda\left(\bei\otimes 1\right)\lambda\left(\ebi\otimes 1\right).
	\end{align}
	In light of the present result, it is natural to ask whether or not there are splittings analogous to ours in each of these settings: in characteristic $p$ for the $\mathfrak{sl}_p$-web algebras, over $\Z$ for the odd arc algebras, and in characteristic 2 for the annular arc algebras, respectively. In the last setting, this would take the form of an algebra isomorphism $\lambda^\alpha:\widetilde{H}_n^\alpha\otimes A_\alpha\to H_n^\alpha$ in characteristic 2 which recovers $\lambda$ for $\alpha_0=\alpha_1=0$.
	\bibliographystyle{alpha}
	\bibliography{ESbib}
\end{document}